\documentclass{amsart}\usepackage[french]{babel}\usepackage[applemac]{inputenc}\usepackage{mathrsfs,amsmath,amssymb}\usepackage[all]{xy}

\newtheorem{theorem}{Th\'eor\`eme}\numberwithin{theorem}{section}\newtheorem{corollary}[theorem]{Corollaire}\newtheorem{conjecture}[theorem]{Conjecture}\newtheorem{lemma}[theorem]{Lemme}\theoremstyle{definition}\newtheorem{remark}[theorem]{Remarque}\newtheorem{definition}[theorem]{D\'efinition}\numberwithin{equation}{section}

\newfont{\cyrm}{wncysc10}\newcommand{\sha}{\text{\cyrm{X}}}\newfont{\cyrma}{wncysc10 at 8pt}\newcommand{\shaa}{\text{\cyrma{X}}}

\begin{document}
\title[D\'ecompte dans une conjecture de Lang]{D\'ecompte dans une conjecture de Lang\\ sur les corps de fonctions : cas des courbes}
\author{Amílcar Pacheco}
\address{Universidade Federal do Rio de Janeiro, Instituto de Matemática, rua Alzira Bran\-dão 355/404, Tijuca, 20550-035 Rio de Janeiro, RJ, Brasil.}
\email{amilcar@acd.ufrj.br}
\author{Fabien Pazuki}
\address{Department of Mathematical Sciences, University of Copenhagen,
Universitetsparken 5, 
2100 Copenhagen \O, Denmark
et 
Institut de Math\'ematiques de Bordeaux, Universit\'e de Bordeaux,
351, cours de la Lib\'eration, 
33405 Talence, France.}
\email{fabien.pazuki@math.u-bordeaux.fr}
\thanks{Le premier auteur a \'et\'e partiellement soutenu par la bourse de recherche CNPq 306045/2013-3. Le second auteur a \'et\'e partiellement soutenu par ANR-10-BLAN-0115 Hamot et par ANR-10-JCJC-0107 Arivaf et est d\'esormais soutenu par la Chaire Niels Bohr DNRF de Lars Hesselholt et l'ANR-14-CE25-0015 Gardio.}
\date{\today}

\begin{abstract}
On se donne une courbe $X$ de genre $d$ sup\'erieur ou \'egal \`a $2$ d\'efinie sur un corps global de fonctions $K$ en caract\'eristique $p>0$ avec $p>2d+1$. On suppose que cette courbe est non-isotriviale. On se donne un sous-groupe $\Gamma$ de $J(K_s)$, o\`u $J$ est la jacobienne de $X$ et $K_s$ une cl\^oture s\'eparable de $K$, v\'erifiant $\Gamma/p \Gamma$ fini. Alors on montre que $X\cap \Gamma$ est de cardinal fini et born\'e par un majorant explicite. Ceci g\'en\'eralise un r\'esultat de Buium et Voloch.\medskip

\noindent\textbf{Abstract.} -- Let $X$ be a non-isotrivial curve of genus $d\geq2$ defined over a global function field $K$ of characteristic $p>0$ with $p>2d+1$. Let $\Gamma$ be a subgroup of $J(K_s)$, where $J$ is the Jacobian variety of $X$ and $K_s$ is a separable closure of $K$, such that $\Gamma/p\Gamma$ is finite. We show that $X\cap \Gamma$ is finite and provide an explicit bound on the number of elements in this intersection. It generalizes a result of Buium and Voloch.
\end{abstract}
\maketitle

\section{Introduction}
Soit $p$ un nombre premier et $n\geq1$ un entier naturel. Soit $k$ un corps fini \`a $q=p^n$ \'el\'ements, $K/k$ un corps de fonctions en une variable et $X/K$ une courbe lisse, compl\`ete et g\'eom\'etriquement connexe de genre $d\ge2$. On dit que $X$ est \emph{isotriviale} s'il existe une extension finie $l$ de $k$, une courbe lisse, compl\`ete et g\'eom\'etriquement connexe $X_0$ d\'efinie sur $l$ et une extension commune $L$ de $l$ et $K$ telle que $X\times_KL\cong X_0\times_lL$. Prenons $X$ non isotriviale. D'apr\`es un r\'esultat de Samuel (prolongeant un th\'eor\`eme de Grauert \cite{Gra}, lequel traite de corps de fonctions en caract\'eristique nulle), l'ensemble $X(K)$ de points $K$-rationnels de $X$ est fini (\textit{confer} \cite{Sam}). 

Notons par $J$ la vari\'et\'e jacobienne de $X$ et par $K_s$ une cl\^oture s\'eparable de $K$. Soit $\Gamma$ un sous-groupe de $J(K_s)$ tel que $\Gamma/p\Gamma$ soit fini. Dans un article ant\'erieur, sous l'hypoth\`ese additionnelle que $X$ n'est pas d\'efinie sur $K^p$, Buium et Voloch ont obtenu une borne sup\'erieure explicite pour le nombre de points dans l'intersection $X\cap\Gamma$, o\`u $X$ est plong\'ee dans $J$ par $\imath:X\hookrightarrow J$ (\textit{confer} \cite{BuVo}). Notons que la condition que $X$ n'est pas d\'efinie sur $K^p$ implique que $X$ ne peut pas \^etre isotriviale.

Soit $(\tau,B)$ la $K/k$-trace de $J$. Rappelons que $(\tau,B)$ est un objet final de la cat\'egorie de paires $(\sigma,A)$, o\`u $A$ est une vari\'et\'e ab\'elienne d\'efinie sur $k$ et en notant $A_K=A\times K$ on demande que $\sigma:A_K\to J$ soit un $K$-homomorphisme entre vari\'et\'es ab\'eliennes. De plus, il existe une \emph{$K/k$-sous-vari\'et\'e ab\'elienne maximale} $J_1$ de $J$ telle que $\tau:B_K\to J$ induit une isog\'enie $\tau_1:B_K\to J_1$. Observons que de ce fait, nous concluons que $J_1$ a partout bonne r\'eduction. Par ailleurs, nous avons aussi que $\mathrm{Tr}_{K/k}(J/J_1)=0$.

Le corps $K$ est le corps de fonctions d'une courbe lisse, compl\`ete, g\'eom\'etriquement connexe $\mathcal C$ d\'efinie sur $k$.

Soit $B_{\mathcal C}=B\times_k\mathcal C$, on dit que ce sch\'ema ab\'elien est un \emph{sch\'ema iso-constant}. Soit $\mathcal J_1/\mathcal C$ le mod\`ele de N\'eron de $J_1/K$ sur $\mathcal C$. L'application $\tau$ induit un homomorphisme de $\mathcal C$-sch\'emas ab\'eliens $\tilde\tau_1:B_{\mathcal C}\to\mathcal J_1$. Soit $d_\tau$ le degr\'e de la diff\'erente de $\tilde\tau_1$ (\emph{confer} \cite[page 203]{Ra}). Observons que par la remarque \ref{remarkdegree}, nous avons $d_{\tau}\le p^{2d_0}\le p^{2d}$, o\`u $d_0=\dim(B)\le d=\dim(J)$.

Soit maintenant $e\ge0$ le plus grand entier tel que $X$ soit d\'efinie sur $K^{p^e}$, mais pas sur $K^{p^{e+1}}$. Notons $\mathcal U$ la sous-courbe affine de $\mathcal C$ o\`u $X$ a partout bonne r\'eduction.\footnote{Rappelons que cela signifie qu'il existe un ensemble fini de points $\mathcal S$ de $\mathcal C$ o\`u $X$ a r\'eduction singuli\`ere, on prend $\mathcal U=\mathcal C\setminus\mathcal S$.}  L'entier $p^e$ correspond au degr\'e d'ins\'eparabilit\'e de l'application $j:\mathcal U\to\mathscr M_d$, o\`u $\mathscr M_d$ d\'esigne l'espace de modules fin des courbes de genre $d$. Le but de cet article est d'\'etendre le r\'esultat de Buium et Voloch, pr\'esent\'e dans la section suivante et d\'emontr\'e dans \cite{BuVo}, au cas o\`u $X$ peut \^etre d\'efinie sur un corps du type $K^{p^n}$ pour un entier $n\ge1$. En effet, on traite le cas maximal, \textit{i.e.} celui o\`u $X$ est d\'efinie sur $K^{p^e}$.

Soit $E_\Gamma$ la plus petite extension alg\'ebrique de $K$ telle que les points de $X\cap\Gamma$ soient rationnels sur $E_\Gamma$. Cette extension est en r\'ealit\'e une extension finie de $K$ (voir sous-section 5.2). Donc, elle correspond \`a un corps de fonctions sur $k$, et on note $g_\Gamma$ son genre. Le but de ce travail est de d\'emontrer le th\'eor\`eme suivant :

\begin{theorem}\label{papa principal}
Soit $X$ une courbe lisse, compl\`ete, g\'eom\'etriquement connexe de genre $d\ge2$ d\'efinie sur $K$ et non isotriviale. Soit $e$ l'entier naturel tel que $X$ est d\'efinie sur $K^{p^e}$ mais pas sur $K^{p^{e+1}}$. Soit $\Gamma$ un sous-groupe de $J(K_s)$ tel que $\Gamma/p\Gamma$ soit fini. Supposons de plus que $p>2d+1$. L'intersection de $X\cap\Gamma$ est finie, de cardinal born\'e de la mani\`ere suivante :
$$\#(X\cap\Gamma)\le C_{\mathrm{BV}}\cdot c_6^e,\quad\text{o\`u}$$
$$\begin{aligned}
C_{\mathrm{BV}}&=\#(\Gamma/p\Gamma)\cdot (3p)^d\cdot (8d-2)\cdot (d!),\\
c_6&=2d\cdot q^{g_\Gamma-1+p\cdot c_5},\\ 
c_5&=[E_\Gamma:K]\cdot\left(p^e\cdot\left(\frac{d}2\cdot(2g_\Gamma+f_{\mathcal X/\mathcal C})+p^{2d}\right)+d\cdot2^{4d^2}\cdot f_{\mathcal X/\mathcal C}\right).
\end{aligned}$$
Si l'extension $E_\Gamma/K$ est mod\'erement ramifi\'ee, alors on peut remplacer $c_5$ par une borne $c_{5,t}$, ainsi que $c_6$ par $c_{6,t}$, qui ne d\'ependent plus de $g_\Gamma$ (voir remarque \ref{remtame}). Si $\mathrm{Tr}_{K/k}(J)=0$, on peut effacer le facteur $p^{2d}$. Le terme $f_{\mathcal X/\mathcal C}$ est le conducteur d'un mod\`ele $\phi:\mathcal X\to\mathcal C$ de $X/K$ sur $\mathcal C$ (voir le paragraphe \ref{conducteur d'une courbe}).
\end{theorem}

\begin{remark}\label{rem1}
La d\'ependance en les param\`etres de d\'epart est partag\'ee entre les donn\'ees relatives au corps de base $p, q, g$, les donn\'ees relatives \`a la courbe $d, e, f_{\mathcal X/\mathcal C}$. 

Le point central ici est le fait qu'on n'impose pas $e=0$. La pr\'esence du degr\'e $[E_\Gamma:K]$, qui d\'epend bien sûr de $\Gamma$, mais aussi de $X$, est peut-\^etre superflue. C'est plutôt une cons\'equence de la m\'ethode de d\'emonstration du th\'eor\`eme. Pour passer de l'hypoth\`ese de Buium et Voloch \cite{BuVo}, \emph{i.e.}, que $X$ n'est pas d\'efinie sur $K^p$ \`a l'hypoth\`ese moins restrictive que $X$ est non isotriviale, on a besoin de faire une descente par Frobenius en caract\'eristique $p$. L'un des outils employ\'es pour ce faire est un groupe de Selmer qui a besoin d'un corps de rationalit\'e des points pour être d\'efini, dans ce cas $E_\Gamma$.

Nous fixons un mod\`ele $\mathcal X/\mathcal C$ de $X/K$ sur $\mathcal C$. Son conducteur est not\'e par  $f_{\mathcal X/\mathcal C}$ (voir la section suivante). Comme analys\'e dans la remarque 1.2 page 3 de \cite{PaPa}, la borne doit d\'ependre de $f_{\mathcal X/\mathcal C}$. Le conducteur $f_{\mathcal X/\mathcal C}$ d\'epend du mod\`ele $\phi:\mathcal X\to\mathcal C$ de $X/K$.

Enfin la borne doit d\'ependre de $\Gamma$. En effet si $(K_m)_{m\in{\mathbb{N}}}$ est une tour d'extensions s\'eparables de $K$ et telle que $K_m$ reste un corps de fonctions en une variable sur $k$ \`a chaque \'etage $m$, en posant $\Gamma_m=J(K_m)$ (qui v\'erifie bien que $\Gamma_m/p\Gamma_m$ est fini), on obtient que $\#(X\cap \Gamma_m)=\#X(K_m)$ est fini pour tout $m$, mais ce cardinal tend vers l'infini avec $m$. Pour obtenir de tels $K_m$, il suffit de choisir $K_m=K(J[\ell^m])$, l'extension galoisienne de $K$ engendr\'ee par les points de $\ell^m$-torsion de $J$, o\`u $\ell\ne p$ est un nombre premier.
\end{remark}

\begin{remark}
En comparaison avec l'article ant\'erieur \cite{PaPa}, la borne pr\'e\-sen\-t\'ee ici est plus fine que celle obtenue sur le cardinal de $X(E_\Gamma)$, la diff\'erence se situant dans la partie de la borne not\'ee $C_{\mathrm{BV}}$ (qui provient de \cite{BuVo}). Dans le cas trait\'e dans \cite{PaPa} la quantit\'e $C_{BV}$ d\'epend du rang de la vari\'et\'e jacobienne $J$ de $X$ sur $E_\Gamma$. Dans le cas pr\'esent elle d\'epend du cardinal du groupe $\Gamma/p\Gamma$. La situation pr\'esente nous permet ainsi de traiter le cas d'autres groupes $\Gamma$ sans pour autant devoir imposer $e=0$, hypoth\`ese qui \'etait faite dans \cite{BuVo}.

 La borne pr\'esent\'ee dans le th\'eor\`eme \ref{papa principal} permet de plus d'obtenir les bornes des corollaires \ref{ManinMumford} et \ref{dynamique} ci-apr\`es.

Si $\Gamma$ est de type fini, on pourra rendre rationnels un nombre fini de g\'en\'erateurs de $\Gamma$, ainsi $E_{\Gamma}$ ne d\'ependra que de $\Gamma$ et pas de son intersection avec $X$.
\end{remark}

\begin{remark}
L'hypoth\`ese $p> 2d+1$ est due \`a l'utilisation du \cite[th\'eor\`eme 5.3]{HiPa}. Cette hypoth\`ese a une triple utilit\'e dans le travail cit\'e. D'abord, si $A$ est une vari\'et\'e ab\'elienne, elle implique que l'extension $K(A[\ell])/K$ est mod\'er\'ement ramifi\'ee (cette extension est engendr\'ee par les coordonn\'ees des points de $\ell$-torsion de $A$, pour un premier $\ell\ne p$, voir \cite{gr}). D'autre part elle implique que le conducteur sauvage de $A/K$ est nul (voir \cite{serre}). De plus la vari\'et\'e ab\'elienne $A$ a partout r\'eduction semi-ab\'elienne sur $K(A[\ell])$ (voir \cite{gr}). Le th\'eor\`eme 5.3 de \cite{HiPa} se situe exactement dans ce cadre, car en partant du sch\'ema semi-ab\'elien universel (contenu dans une compactification bien choisie de l'espace de modules de vari\'et\'es ab\'eliennes principalement polaris\'ees avec une structure de niveau convenable), on construit un mod\`ele semi-ab\'elien $\psi:\mathcal B\to\mathcal C$ de $A/K$ sur $\mathcal C$, dont les diff\'erentielles sont images inverses de celles du sch\'ema semi-ab\'elien universel. C'est ainsi que fonctionne la preuve du th\'eor\`eme cit\'e (voir \cite[Theorem 3.1]{esvi}). Il appara\^it donc difficile de se passer de cette hypoth\`ese.
\end{remark}

\begin{remark}
Avant de passer aux corollaires, nous attirons l'attention du lecteur sur le fait que la section 4 propose la preuve d'une in\'egalit\'e abc pour les vari\'et\'es ab\'eliennes en caract\'eristique $p>0$ qui constitue une g\'en\'eralisation non triviale du r\'esultat ant\'erieur de \cite{HiPa}. Ce r\'esultat sera sans doute utile \`a d'autres endroits dans l'\'etude de l'arithm\'etique des vari\'et\'es ab\'eliennes sur un corps de fonctions sur un corps fini. Voici son \'enonc\'e.
\end{remark}

\begin{theorem}\label{thmabc}
Soit $A/K$ une variété abélienne non constante de dimension $d$. Supposons que $p>2d+1$. Soit $\bar s$ le nombre de points g\'eom\'etriques de $\mathcal C$ o\`u $A/K$ a mauvaise r\'eduction. Alors,
$$h_{\mathrm{dif}}(A/K)\le p^e\cdot\left(\frac{d-d_0}2\cdot(2g-2+\bar s)+p^{2d}\right)+d\cdot2^{4d^2}\cdot\bar s,$$
et \textbf{a fortiori} (\textbf{cf.} remarque \ref{rems}),
$$h_{\mathrm{dif}}(A/K)\le p^e\cdot\left(\frac{d-d_0}2\cdot(2g-2+f_{A/K})+p^{2d}\right)+d\cdot2^{4d^2}\cdot f_{A/K}.$$
Si la $K/k$-trace de $A$ est nulle, on peut effacer le terme $p^{2d}$ de la borne.
\end{theorem}

\begin{remark}\label{remMB}
Pour un exemple o\`u la $K/k$-trace d'une vari\'et\'e  ab\'elienne n'est pas nulle, nous renvoyons au travail de Moret-Bailly \cite{mbex}. En effet, dans son exemple $d=\dim(A)=2$ et $d_0=\dim(B)=2$, $A$ a partout bonne r\'eduction sur $\mathbb \mathbb k(t)$ (sans pourtant \^etre une vari\'et\'e ab\'elienne constante), o\`u $k$ est un corps fini, comme au d\'ebut de l'introduction. Nous avons un sch\'ema ab\'elien $\mathcal A\to\mathbb P^1$ qui n'est pas iso-constant. Pour cet exemple $h_{\mathrm{dif}}(A/K)=p$ et $e=1$.  Donc, le th\'eor\`eme \ref{thmabc} nous donne $p\le p^5$, ce qui est bien v\'erifi\'e.
 \end{remark}

On donne maintenant des corollaires du th\'eor\`eme \ref{papa principal}. Pour une vari\'et\'e ab\'elienne $A$, notons $A_{\text{tors}}$ l'ensemble de ses points de torsion et $A_{p'\text{-tors}}$ l'ensemble des points de torsion d'ordre premier \`a $p$. Si on sp\'ecialise $\Gamma=J(K_s)_{p'\text{-tors}}$, on obtient en corollaire du th\'eor\`eme \ref{papa principal} une borne explicite sur le probl\`eme de Manin-Mumford en caract\'eristique $p$, g\'en\'eralisant l'article \cite{vo}, lequel montrait la finitude (dans le cas $J$ ordinaire et $X$ non d\'efini sur $K^p$), mais ne donnait pas de borne. Le th\'eor\`eme \ref{papa principal} nous garantit l'existence d'une extension finie $K'/K$ telle que $X\cap J(K_s)_{p'\text{-tors}}\subset X(K')$. Soit $g'$ le genre de $K'$.

\begin{corollary}\label{ManinMumford}
Soit $X$ une courbe lisse, compl\`ete, g\'eom\'etriquement connexe de genre $d\ge2$, non isotriviale, et d\'efinie sur $K^{p^e}$ mais pas sur $K^{p^{e+1}}$.  Supposons de plus que $p>2d+1$. Le nombre de points de $p'$-torsion de $J(K_s)$ qui sont sur $X$ est fini et born\'e par :
$$\#(X\cap J(K_s)_{p'\text{-}\mathrm{tors}})\le (3p)^d\cdot (8d-2)\cdot (d!)\cdot c_8^e,\quad\text{o\`u}$$
$$c_8=2d\cdot q^{g'-1+p\cdot c_7},\quad\text{et}$$
$$c_7=[K':K]\cdot\left(p^e\cdot\left(\frac{d}2\cdot(2g'+f_{\mathcal X/\mathcal C})+p^{2d}\right)+d\cdot2^{4d^2}\cdot f_{\mathcal X/\mathcal C}\right).$$
Si l'extension $K'/K$ est mod\'erement ramifi\'ee, on peut se passer du genre $g'$ de $K'$ dans les formules ant\'erieures (voir remarque \ref{remtame}). Si la $K/k$-trace de $J$ est nulle, on peut effacer le terme $p^{2d}$ de la somme ant\'erieure.
\end{corollary}

\begin{proof}
Il suffit de borner l'intersection par le th\'eor\`eme \ref{papa principal} appli\-qu\'e \`a $\Gamma=J(K_s)_{p'\text{-tors}}$, lequel v\'erifie $\Gamma/p\Gamma=\{0\}$. 
\end{proof}

\begin{remark}
Pour la finitude de l'intersection $J[p^{\infty}]\cap J(K_s)$, voir les r\'esultats (valables g\'en\'eriquement) de l'article \cite{Vol95} section 4 page 1092.
\end{remark}

On donne ensuite un autre corollaire du th\'eor\`eme \ref{papa principal} concernant l'intersection d'une courbe et des multiples d'un point rationnel.

\begin{corollary}\label{dynamique}
Soit $X$ une courbe lisse, compl\`ete, g\'eom\'etriquement connexe de genre $d\ge2$, non isotriviale, d\'efinie sur $K^{p^e}$ mais pas sur $K^{p^{e+1}}$. Supposons de plus que $p>2d+1$. Soit $P\in{X(K)}$. Alors on a la borne
$$\#(X\cap (\mathbb{Z}\cdot P))\le p\cdot(3p)^d\cdot (8d-2)\cdot (d!)\cdot c_{10}^e,\quad\text{o\`u}$$
$$c_{10}=2d\cdot q^{g-1+p\cdot c_9},\quad\text{et}$$
$$c_9=p^e\cdot\left(\frac{d}2\cdot(2g+f_{\mathcal X/\mathcal C})+p^{2d}\right)+d\cdot2^{4d^2}\cdot f_{\mathcal X/\mathcal C}.$$
\end{corollary}

\begin{proof}
Il suffit d'observer que $\Gamma=\mathbb{Z}\cdot P$ v\'erifie $\#\Gamma/p\Gamma\leq p$. 
\end{proof}

Nous allons suivre le plan suivant. Apr\`es l'expos\'e de deux r\'esultats ant\'erieurs reli\'es \`a cette question en partie 2, nous d\'ecrirons dans la partie 3 les objets utiles \`a la preuve, notamment le morphisme de Frobenius relatif $F$, les groupes de Selmer et les conducteurs de courbes et de vari\'et\'es ab\'eliennes. En partie 4 on prouvera le th\'eor\`eme \ref{thmabc}, un abc pour les vari\'et\'es ab\'eliennes en caract\'eristique $p>0$. En partie 5 on s'int\'eressera \`a d\'ecrire les groupes de Selmer locaux dans les cas de bonne r\'eduction potentielle et de r\'eduction semi-ab\'elienne potentielle. En partie 6 on montrera comment passer au groupe de Selmer global pour mener \`a bien une $F$-descente et conclure la preuve. Enfin, la derni\`ere partie sera consacr\'ee \`a une comparaison avec le cas des corps de nombres o\`u un r\'esultat de comptage a \'et\'e obtenu par G. R\'emond.

\subsection{Remerciements.} Nous remercions D. Vauclair de nous avoir pos\'e une question qui nous a permis dans la sous-section \ref{frob} de compl\'eter la pr\'esentation de la preuve. Dans la section 3 on remarquera que dans la preuve du r\'esultat ant\'erieur \cite[theorem 1.1]{PaPa}, on traitait uniquement le cas o\`u la jacobienne $J$ est ordinaire. L'ordinarit\'e de $J$ \'equivaut \`a $\dim_{\mathbb F_p}J[p]=d$.  Les autres cas se traitent en fait d'une mani\`ere similaire (voir sections 3 et 4 du pr\'esent texte) et l'\'enonc\'e \cite[theorem 1.1]{PaPa} reste vrai tel qu'il est. Nous remercions Felipe Voloch pour ses commentaires, nous permettant de corriger le corollaire \ref{ManinMumford}. 

\section{Description des r\'esultats ant\'erieurs}
\begin{theorem}[Buium-Voloch]\label{bv} \cite[Theorem]{BuVo}
Soit $X/K$ une courbe lisse, compl\`ete, g\'eom\'etriquement connexe de genre $d\ge2$, soit $\Gamma$ un sous-groupe de $J(K_s)$ tel que $\Gamma/p\Gamma$ est fini. On suppose que $X$ ne soit pas d\'efinie sur $K^p$, on a alors
$$\#(X\cap\Gamma)\le\#(\Gamma/p\Gamma) (3p)^d(8d-2) (d!).$$
\end{theorem}

Notre but est de relâcher la condition selon laquelle $X$ n'est pas d\'efinie sur $K^p$, qui est superflue, tout en conservant bien s\^ur $X$ non isotriviale car c'est une hypoth\`ese n\'ecessaire. Le premier r\'esultat dans cette direction provient de \cite{PaPa} :

\begin{theorem}[Pacheco-Pazuki]\label{papa} \cite[Theorem]{PaPa}
Soit $X/K$ une courbe lisse, projective, g\'eom\'etriquement connexe d\'efinie sur $K$ et de genre $d\ge2$. On suppose que $X$ est non isotriviale.  On suppose de plus que $p>2d+1$.

 Soit $e$ le plus grand entier naturel tel que $X$ est definie sur $K^{p^e}$ mais pas sur $K^{p^{e+1}}$, alors 
$$\#X(K)\leq  C_{\mathrm{BV}}'\cdot  C_{\mathrm{desc}}\sp e,\quad\text{o\`u}$$
$$\begin{aligned}
C_{BV}'&=p^{2d\cdot(2g+1)+f_{\mathcal X/\mathcal C}} \cdot3^d \cdot(8d-2)\cdot d!\quad\mathrm{et}\\
C_{\mathrm{desc}}&=q^{c_0}\quad\mathrm{avec}\quad c_0=g-1+f_{\mathcal X/\mathcal C}+\frac{1}{2}\cdot p^{e+1} \cdot d\cdot (2g-2+2^{4d^2}\cdot f_{\mathcal X/\mathcal C}).
\end{aligned}$$
\end{theorem}

Notons que ces deux \'enonc\'es sont de nature diff\'erente. Le premier concerne des points alg\'ebriques sur une courbe, il y en a une infinit\'e. On obtient la finitude en intersectant avec un sous-groupe de la jacobienne associ\'ee. Le second compte des points rationnels sur une courbe, dont on sait qu'ils sont en nombre fini par le th\'eor\`eme de Samuel. Il y a toutefois un lien entre les deux : pour d\'emontrer le th\'eor\`eme \ref{papa}, on commence par sp\'ecialiser $\Gamma=J(K)\subset J(K_s)$ dans le th\'eor\`eme \ref{bv}, ce qui a pour effet de concentrer la recherche sur les points $K$-rationnels ; on fait ensuite fonctionner une $F$-descente en caract\'eristique $p$ pour lever l'hypoth\`ese $X$ non d\'efinie sur $K^p$. Le r\'esultat pr\'esent\'e ici g\'en\'eralise donc \`a la fois le th\'eor\`eme \ref{bv} et le th\'eor\`eme \ref{papa} en adaptant la strat\'egie de descente \`a la situation g\'en\'erale. 

\newpage

\section{Pr\'eliminaires}
\subsection{Noyau de Frobenius}\label{frob}
Soit $A/K$ une vari\'et\'e ab\'elienne non constante de dimension $d$. Soit $F_{\mathrm{abs}}$ l'automorphisme de Frobenius absolu d\'efini sur $K$ et $A^{(p)}/K$ la vari\'et\'e ab\'elienne d\'efinie par le diagramme suivant : 
\begin{equation}\label{frobvar}
\xymatrix{A^{(p)}\ar[r]\ar[d]&A\ar[d]\\ \mathrm{Spec}\,K\ar[r]^{F_{\mathrm{abs}}}&\mathrm{Spec}\,K.}
\end{equation}
Ce diagramme induit un \emph{morphisme de Frobenius relatif} (qui est une isog\'enie purement ins\'eparable) $F:A\to A^{(p)}$. Il existe une isog\'enie compl\'ementaire $V:A^{(p)}\to A$ telle que $V\circ F=[p]_A$ et $F\circ V=[p]_{A^{(p)}}$. Cette isog\'enie $V$ est appel\'ee le \emph{Verschiebung}. Notons par $\mu_p$, respectivement par $\alpha_p$ le noyau du morphisme de Frobenius absolu $F_{\mathrm{abs}}$ sur $\mathbb G_m(K)$, respectivement sur $\mathbb G_a(K)$. Ce sont des sch\'emas en groupes plats d'ordre $p$. En tant que sch\'ema en groupes le noyau de $F$ est d\'ecrit de la façon suivante : 
\begin{equation}\label{frobenius}
\ker F=\mu_p^{\oplus a}\oplus G,
\end{equation}
o\`u $0\le a\le d$ est un entier (\textit{confer} \cite[\S15]{Mum}) et $G$ est un sch\'ema en groupe de type local-local. De plus, $G$ admet une s\'erie de composition de sch\'emas en groupes
$$G=G_0\supset G_1\supset\cdots\supset G_{2d-a}=0,$$
telle que chaque quotient $G_i/G_{i+1}$ est isomorphe \`a $\alpha_p$. Lorsque $a=d$ et $G=0$, la vari\'et\'e ab\'elienne $A$ est dite ordinaire.

\begin{remark}\label{rempapa}
Dans \cite[3.4.2]{PaPa}, la preuve ne traitait implicitement que le cas o\`u on fait cette hypoth\`ese d'ordinarit\'e, mais vaut en fait en toute g\'en\'eralit\'e comme on va le voir dans la partie suivante.
\end{remark}

\subsection{Description des groupes de Selmer}
(Voir \cite[\S1]{ul}.) Soit $\varphi:B\to A$ une isog\'enie entre vari\'et\'es ab\'eliennes d\'efinies sur $K$. On consid\`erera tous les groupes de cohomologie calcul\'es dans le petit site plat $K_{\mathrm{fl}}$ de $\mathrm{Spec}\,K$. On a donc une suite exacte de sch\'emas en groupes 
$$0\to\ker\varphi\to B\to A\to0.$$
Pour toute place $v$ de $K$, notons par $K_v$ le compl\'et\'e de $K$ en $v$. L'image de l'application cobord (provenant de la suite de cohomologie longue associ\'ee \`a la suite courte ant\'erieure) $\delta_v:A(K_v)\to H^1(K_v,\ker\varphi)$ est d\'efinie comme le groupe de Selmer local $\mathrm{Sel}_B(K_v,\varphi)$. Le groupe de Selmer global $\mathrm{Sel}_B(K,\varphi)$ est d\'efini comme le sous-groupe de $H^1(K,\ker\varphi)$ form\'e des classes $\xi$ telles que ses restrictions locales $\xi_v$ tombent dans $\mathrm{Sel}_B(K_v,\varphi)$. 

Le groupe de Selmer est reli\'e au groupe de Tate-Shafarevich 
$$\sha(B/K)=\ker\left(H^1(B,K)\to\prod_{v\in{M_K}}H^1(K_v,B)\right),$$ o\`u $M_K$ d\'esigne l'ensemble des places de $K$, 
de la façon suivante. Notons par $\varphi_{\shaa}:\sha(B/K)\to\sha(A/K)$ l'application induite par $\varphi$. On a donc une suite exacte de groupes 
$$0\to A(K)/\varphi(B(K))\to\mathrm{Sel}_B(K,\varphi)\to\ker\varphi_{\shaa}\to0.$$
En pratique $\mathrm{Sel}_B(K,\varphi)$ est fini et effectivement calculable.

Les propri\'et\'es suivantes sont prouv\'ees dans \cite[\S1]{ul} :
\begin{itemize}
\item[$\bullet$] Soit $\mathcal O_v$ l'anneau de valuation de $K_v$. Supposons que $B$ et $A$ aient bonne r\'eduction en $v$. Alors la restriction de l'application 
\begin{equation}\label{prop1}
H^1(\mathcal O_v,\ker\varphi)\to H^1(K_v,\ker\varphi)
\end{equation}
induit un isomorphisme 
\begin{equation}\label{prop2}
H^1(\mathcal O_v,\ker\varphi)\cong\mathrm{Sel}_B(K_v,\varphi) .
\end{equation}

\item[$\bullet$] Si $L_w$ est une extension galoisienne finie de $K_v$ de groupe $G(w|v)$ d'ordre premier \`a $\deg\varphi$ (o\`u $w$ est une place au-dessus de $v$), l'application d'inclusion 
\begin{equation}\label{prop3}
H^1(K_v,\ker\varphi)\to H^1(L_w,\ker\varphi)
\end{equation}
induit un isomorphisme 
\begin{equation}\label{prop4}
\mathrm{Sel}_B(K_v,\varphi)\cong\mathrm{Sel}_B(L_w,\varphi)^{G(w|v)}.
\end{equation}
\item[$\bullet$] De façon similaire, si $L$ est maintenant une extension galoisienne finie de $K$ de groupe $G(L/K)$ d'ordre premier \`a $\deg\varphi$, nous avons un isomorphisme 
\begin{equation}\label{SelmerG}
\mathrm{Sel}_B(K,\varphi)\cong\mathrm{Sel}_B(L,\varphi)^{G(L/K)}.
\end{equation}
\end{itemize}

\subsection{Diff\'erentielles et diviseurs}
Fixons une vari\'et\'e ab\'elienne $A/K$ de dimension $d$ d\'efinie sur $K$. Soit $\phi:\mathcal A\to\mathcal C$ le mod\`ele de N\'eron de $A/K$ sur $\mathcal C$ et $e_{\mathcal A}:\mathcal C\to\mathcal A$ sa section neutre. Soit $\omega_{\mathcal A/\mathcal C}=e_{\mathcal A}^*(\bigwedge^d\,\Omega^1_{\mathcal A/\mathcal C})$, c'est un faisceau inversible sur $\mathcal C$ qui correspond \`a un diviseur $\mathcal D(\omega_{\mathcal A/\mathcal C})$ de $\mathcal C$. D\'efinissons la hauteur diff\'erentielle de $A/K$ par 
$$h_{\mathrm{dif}}(A/K)=\deg(\omega_{\mathcal A/\mathcal C}).$$ 

\subsection{Conducteur d'une vari\'et\'e ab\'elienne}
Soit $A/K$ une vari\'et\'e ab\'elienne de dimension $d$, soit $v$ une place de $K$ et $I_v$ un groupe d'inertie en $v$. Notons par $\ell\ne p$ un nombre premier. Soit $T_\ell(A)$ le module de Tate $\ell$-adique de $A$ et $V_\ell(A)=T_\ell(A)\otimes\mathbb Q_\ell$. Soit $\epsilon_v=\mathrm{codim}\,V_\ell(A)^{I_v}$. Sous l'hypoth\`ese que $p>2d+1$, il n'y pas de contribution de la ramification sauvage pour le conducteur de $A/K$ (voir \cite{gr}), donc le conducteur de $A/K$ est d\'efini par 
$$\mathfrak F_{A/K}=\sum_{v\in{M_K}}\epsilon_v\cdot[v]\quad\text{ et }\quad f_{A/K}=\deg\mathfrak F_{A/K}.$$

\begin{remark}\label{rems}
Soit $\bar s$ le nombre de points g\'eom\'etriques de $\mathcal C$ o\`u $A/K$ a mauvaise r\'eduction. Observons que $\bar s\le f_{A/K}$.
\end{remark}

\subsection{Conducteur d'une courbe}\label{conducteur d'une courbe}
Soit $X/K$ une courbe lisse, compl\`ete et g\'e\-om\'e\-triquement connexe. Notons par $\phi:\mathcal X\to\mathcal C$ un mod\`ele de $X/K$ sur $\mathcal C$. Pour tout $v\in\mathcal C$, notons par $\mathcal X_{v,\bar\kappa_v}$ la fibre g\'eom\'etrique de $\phi$ en $v$. 

Nous calculons toutes les caract\'eristiques d'Euler-Poincar\'e par rapport \`a la cohomologie $\ell$-adique. Pour $i\in\{0,1,2\}$ notons par $\delta_{i,X,v}$ le conducteur de Swan de $H^i(X_{\bar K},\mathbb Q_\ell)$ en $v$ (voir \cite{gr}). Soit 
$$\chi(\delta_{X,v})=\sum_{i=0}^2(-1)^i\cdot\delta_{i,X,v}.$$
La multiplicit\'e en $v$ du conducteur de $\mathcal X/\mathcal C$ est donn\'ee par :
$$f_{\mathcal X/\mathcal C,v}=-\chi(X_{\bar K})+\chi(\mathcal X_{v,\bar\kappa_v})+\chi(\delta_{X,v}).$$
De plus, le conducteur global est d\'efini par 
$$f_{\mathcal X/\mathcal C}=\sum_{v\in{\mathcal{C}}} f_{\mathcal X/\mathcal C,v}\cdot\deg v.$$ 
Rectifions ici une remarque os\'ee que l'on trouve dans \cite {PaPa} : le conducteur $f_{\mathcal X/\mathcal C}$ est le conducteur du mod\`ele $\mathcal X/\mathcal C$ de $X/K$ sur $\mathcal C$ et n'est donc pas une notion birationnelle.

\section{Un th\'eor\`eme abc pour les vari\'et\'es ab\'eliennes en caract\'eristique $p$}
Dans les trois premi\`eres sous-sections de ce paragraphe, nous supposons que $A/K$ a partout r\'eduction semi-ab\'elienne. Dans la  quatri\`eme sous-section, on supprime cette hypoth\`ese.

Une autre hypoth\`ese n\'ecessaire, initialement, pour les deux premi\`eres sous-sections c'est que \emph{l'application de Kodaira-Spencer} (associ\'ee \`a la restriction d'un mod\`ele de N\'eron $\phi:\mathcal A\to\mathcal C$ d'une vari\'et\'e ab\'elienne $A/K$ \`a l'ouvert $\mathcal U$ de $\mathcal C$ o\`u $A/K$ a bonne r\'eduction) est non nulle. Pour la d\'efinition de cette application, voir \cite[(5.13)]{HiPa}. Rappelons juste que cela \'equivaut \`a dire que $A$ est d\'efinie sur $K$, mais pas sur $K^p$. Apr\`es avoir obtenu la borne sup\'erieure pour la hauteur diff\'erentielle de $A/K$, sous cette hypoth\`ese, on revient au cas g\'en\'eral par un argument associ\'e \`a l'application de Frobenius absolue for $\mathrm{Spec}(K)$ (voir \cite[\S5.5]{HiPa}). 

Pour la trosi\`eme sous-section, l'hypoth\`ese sur le non annuellement de l'application de Kodaira-Spencer n'est pas n\'ecessaire.

\subsection{Cas r\'eduction semi-ab\'elienne et trace nulle}

\begin{theorem}\label{abcthm}\cite[Theorem 5.3]{HiPa}
Soit $A/K$ une vari\'et\'e ab\'elienne non cons\-tan\-te de dimension $d$. Supposons que $A$ ait partout r\'eduction semi-ab\'elienne et que l'application de Kodaira-Spencer associ\'ee \`a son mod\`ele de N\'eron $\phi:\mathcal:\to\mathcal C$ soit nulle. Supposons aussi que la $K/k$-trace de $A$ soit nulle et de plus que $p> 2d+1$. Soit $\bar s$ le nombre des points g\'eom\'etriques de $\mathcal C$ o\`u $A$ a mauvaise r\'eduction. L'in\'egalit\'e suivante est satisfaite : 
$$h_{\mathrm{dif}}(A/K)\le\frac{d}2\cdot(2g-2+\bar s),$$
et, \textbf{a fortiori} (\textbf{cf.} remarque \ref{rems}),
$$h_{\mathrm{dif}}(A/K)\le\frac{d}2\cdot(2g-2+f_{A/K}).$$
\end{theorem}

\begin{remark}\label{remtrace}
Rappelons que la $K/k$-trace $(\tau,B)$ de $A$ consiste d'une vari\'et\'e ab\'elienne $B$ d\'efinie sur $k$ et d'un homomorphisme de $K$-vari\'et\'es ab\'eliennes $\tau:B_K\to A$. Elle satisfait la propri\'et\'e universelle que pour n'importe qu'elle autre vari\'et\'e ab\'elienne $C/k$ et homomorphisme $\tau':C_K\to A$, cet homomorphisme se factorise par $\tau$. 
\end{remark}

\subsection{Cas r\'eduction semi-ab\'elienne et trace non nulle}
En caract\'eristique $p>0$, le morphisme $\tau$ n'est pas n\'ecessairement injectif. En effet, il suit du fait que $k$ est un corps fini que l'extension $K/k$ est r\'eguli\`ere. Donc, $\ker(\tau)$ est un sch\'ema en groupes connexe, \emph{a fortiori} \textit{infinit\'esimal} (\emph{confer} \cite[Theorem 6.12]{Co}). Il existe une $K$-sous-vari\'et\'e ab\'elienne $A_1$ de $A$ telle que $\mathrm{Tr}_{K/k}(A/A_1)=0$. Cette sous-vari\'et\'e ab\'elienne est dite \emph{la sous-vari\'et\'e ab\'elienne maximale de $A$ par rapport \`a $K/k$}. L'homomorphisme $\tau$ induit  une $K$-isog\'enie $\tau_1:B_K\to A_1$. En particulier, $A_1$ a partout bonne r\'eduction.

Soit $\phi:\mathcal A\to\mathcal C$, respectivement $\phi_1:\mathcal A_1\to\mathcal C$, le mod\`ele de N\'eron de $A/K$ sur $\mathcal C$, respectivement de $A_1/K$ sur $\mathcal C$. Notons $B_{\mathcal C}=B\times_k\mathcal C$, nous dirons que $B_{\mathcal C}$ est un sch\'ema ab\'elien \emph{iso-constant}. L'isog\'enie $\tau_1$ s'\'etend en un homomorphisme de $\mathcal C$-sch\'emas ab\'eliens $\tilde\tau_1:B_{\mathcal C}\to\mathcal A_1$, notons $\tilde H=\ker(\tilde\tau_1)$. Observons que $\mathcal A_1$ est un sous-sch\'ema ab\'elien du sch\'ema semi-ab\'elien $\mathcal A$.  On en d\'eduit une suite exacte de $\mathcal C$-sch\'emas en groupes : 
\begin{equation}\label{suite1}
\xymatrix{0\ar[r]&\tilde H\ar[r]&B_{\mathcal C}\ar[r]^{\tilde\tau_1}&\mathcal A\ar[r]&\mathcal A/\mathcal A_1\ar[r]&0.}
\end{equation}

Pour tout sch\'ema en groupes $\mathcal G$ lisse et plat sur $\mathcal C$ de dimension relative $\gamma$, soit $e_{\mathcal G}$ sa section unit\'e et $\omega_{\mathcal G/\mathcal C}=e_{\mathcal G}^*(\bigwedge^\gamma\Omega^1_{\mathcal G/\mathcal C})$. Ce faisceau est inversible sur $\mathcal C$ et son degr\'e est not\'e par $\deg(\omega_{\mathcal G/\mathcal C})$. La suite exacte (\ref{suite1}) nous donne un isomorphisme : 
\begin{equation}\label{iso1}
\omega_{\mathcal A/\mathcal C}\cong\omega_{B_{\mathcal C}/\mathcal C}\otimes\omega_{(\mathcal A/\mathcal A_1)/\mathcal C}\otimes\omega_{\tilde H/\mathcal C}^{-1}.
\end{equation}
(Pour la d\'efinition du dernier terme voir \cite[\S2, p. 36, 2.2 (b)]{De}). D'o\`u : 
\begin{equation}\label{deg1}
\deg(\omega_{\mathcal A/\mathcal C})=\deg(\omega_{B_{\mathcal C}/\mathcal C})+\deg(\omega_{(\mathcal A/\mathcal A_1)/\mathcal C})-\deg(\omega_{\tilde H/\mathcal C}).
\end{equation}
Comme $B_{\mathcal C}/\mathcal C$ est iso-constant, nous concluons que $\deg(\omega_{B_{\mathcal C}/\mathcal C})=0$.

Il suit maintenant du th\'eor\`eme \ref{abcthm} : 
\begin{equation}\label{deg2}
\deg(\omega_{(\mathcal A/\mathcal A_1)/\mathcal C})\le\frac{d-d_0}2\cdot(2g-2+\bar{s}).
\end{equation}

Soit $\mathcal D_{\tilde\tau_1}$ la diff\'erente de $\tilde\tau_1$ d\'efinie par Raynaud en \cite[Proposition 1.4.1, p. 205]{Ra}. Il montre dans ce texte que $\omega_{\tilde H/\mathcal C}^{-1}\cong\mathcal D_{\tilde\tau_1}$.
Donc,
$$\deg(\omega_{\mathcal A/\mathcal C})\le\frac{d-d_0}2\cdot(2g-2+\bar s)+\deg(\mathcal D_{\tilde\tau_1}).$$

\begin{remark}\label{remarkdegree}
Il suit du théorème 2.1.1 de \cite{Ra} que si $\tilde\tau_1^\vee$ note l'isogénie duale de $\tilde\tau_1$, nous avons l'égalité suivante : 
$$\mathcal D_{\tilde\tau_1}\cdot\mathcal D_{\tilde\tau_1^\vee}=(\deg(\tilde\tau_1)).$$
En particulier, 
$$\deg(\mathcal D_{\tilde\tau_1})\le\deg(\tilde\tau_1).$$ 
Par ailleurs, d'après \cite[\S6]{Co}, une fois que $k$ est fini, donc $K/k$ est régulière, nous avons que $\ker(\tilde\tau_1)$ est un schéma en groupes connexe d'ordre $p^{2d_0}$, où $d_0=\dim(B)\le d=\dim(A)$. En conséquence, 
$$\deg(\tilde\omega_{\mathcal A})\le\frac{d-d_0}2\cdot(2g-2+\bar s)+p^{2d}.$$
\end{remark}

\begin{theorem}\label{corabc}
Soit $A/K$ une vari\'et\'e ab\'elienne non constante de dimension $d$ telle que l'application de Kodaira-Spencer associ\'ee \`a son mod\`ele de N\'eron $\phi:\mathcal A\to\mathcal C$ est nulle. Soit $(\tau,B)$ sa $K/k$-trace et $d_0=\dim(B)$. Supposons que $A/K$ ait partout r\'eduction semi-ab\'elienne et que $p>2d+1$. Soit $\bar s$ le nombre de points g\'eom\'etriques de $\mathcal C$ o\`u $A$ admet mauvaise r\'eduction. Alors,
$$h_{\mathrm{dif}}(A/K)\le\frac{d-d_0}2\cdot(2g-2+\bar s)+p^{2d},$$
et \textbf{a fortiori} (\textbf{cf.} remarque \ref{rems}),
$$h_{\mathrm{dif}}(A/K)\le\frac{d-d_0}2\cdot(2g-2+f_{A/K})+p^{2d}.$$

En particulier, si $\mathrm{Tr}_{K/k}(A)=0$ nous pouvons effacer le terme $p^{2d}$ de la borne.
\end{theorem}

\subsection{L'application de Kodaira-Spencer est non nulle}
Dans ce cas-l\`a, on note par $e\ge1$ le plus grand entier tel que $A$ soit d\'efinie sur $K^{p^e}$, mais n'est pas d\'efinie sur $K^{p^{e+1}}$. Donc, il existe une vari\'et\'e ab\'elienne $A_i$ d\'efinie sur $K$, mais pas sur $K^p$ (donc, telle que son application de Kodaira-Spencer associ\'ee est non nulle), telle que $A\cong A_i^{(p^e)}$. Il suit de \cite[Remark 5.13, (5.18)]{HiPa} que 
$$h_{\mathrm{dif}}(A/K)=p^e\cdot h_{\mathrm{dif}}(A_i/K).$$

Il suit donc du th\'eor\`eme \ref{corabc} la g\'en\'eralisation suivante.

\begin{theorem}\label{thmksnon}
Soit $A/K$ une vari\'et\'e ab\'elienne non constante de dimension $d$. Soit $(\tau,B)$ sa $K/k$-trace et $d_0=\dim(B)$. Supposons que $A/K$ ait partout r\'eduction semi-ab\'elienne et que $p>2d+1$. Soit $\bar s$ le nombre de points g\'eom\'etriques de $\mathcal C$ o\`u $A$ admet mauvaise r\'eduction. Alors,
$$h_{\mathrm{dif}}(A/K)\le p^e\cdot\left(\frac{d-d_0}2\cdot(2g-2+\bar s)+p^{2d}\right),$$
et \textbf{a fortiori} (\textbf{cf.} remarque \ref{rems}),
$$h_{\mathrm{dif}}(A/K)\le p^e\cdot\left(\frac{d-d_0}2\cdot(2g-2+f_{A/K})+p^{2d}\right).$$

En particulier, si $\mathrm{Tr}_{K/k}(A)=0$ nous pouvons effacer le terme $p^{2d}$ de la borne.
\end{theorem}

\begin{remark}
Observons que le plus grand entier $e\ge0$ tel que $X$ soit d\'efinie sur $K^{p^e}$, mais ne soit pas d\'efinie sur $K^{p^{e+1}}$, est aussi le plus grand entier tel que sa vari\'et\'e jacobienne $J$ soit d\'efinie sur $K^{p^e}$, mais ne soit pas d\'efinie sur $K^{p^{e+1}}$. Donc, le $e$ qui apparait dans le th\'eor\`eme \ref{corabc}, dans le cas o\`u $A$ est la vari\'et\'e jacobienne de $X$, est le m\^eme que celui d\'efini dans la introduction.
\end{remark}

\subsection{Cas r\'eduction quelconque et trace quelconque}
Soit $A/K$ une vari\'et\'e ab\'elienne de dimension $d$ \`a r\'eductions quelconques. Soit $B=\mathrm{Tr}_{K/k}(A)$ de dimen\-sion $d_0$.

Soit $\ell\ne p$ un premier et $L=K(A[\ell])$. Notons par $\mathcal Ust_{A/K}$ l'ensemble des places $v$ de $K$ telles que $A/K$ n'a pas de r\'eduction semi-ab\'elienne sur $v$. 

\subsubsection{Le conducteur de changement de base}
Pour toute $v\in\mathcal Ust_{A/K}$, soit $w$ une place de $L$ au-dessus de $v$. Soit $\phi:\mathcal A\to\mathcal C$ le mod\`ele de N\'eron de $A/K$ sur $\mathcal C$. Notons par $e_{\mathcal A}$ sa section unit\'e. Soit $\mathcal A_v$ le mod\`ele de N\'eron de $A_{K_v}$ sur $\mathcal O_v$ et $\mathcal A_w$ le mod\`ele de N\'eron de $A_{L_w}$ sur $\mathcal O_w$. Notons
$$\Omega(A,w|v)=\frac{H^0(\mathrm{Spec}(\mathcal O_v),e_{\mathcal A}^*(\bigwedge^d\Omega^1_{\mathcal A_v/\mathcal O_v}))\otimes\mathcal O_w}{H^0(\mathrm{Spec}(\mathcal O_w),e_{\mathcal A}^*(\bigwedge^d\Omega^1_{\mathcal A_w/\mathcal O_w}))}.$$
Cet objet est un $\mathcal O_w$-module de longueur finie not\'ee $l(A,w|v)$. \emph{Le conducteur de changement de base} est d\'efini par 
$$c(A,w|v)=e(w|v)^{-1}l(A,w|v),$$ 
o\`u $e(w|v)$ note l'indice de ramification de $w$ sur $v$. Rappelons que la hauteur stable de $A/K$ est donn\'ee par la formule 
$$h_{\mathrm{st}}(A)=[L:K]^{-1}h_{\mathrm{dif}}(A_L/L).$$ 
Il suit de la d\'emonstration de \cite[Lemma 3.4]{HiPa} que 
\begin{equation}\label{ht1}
h_{\mathrm{st}}(A)+\sum_{v\in\mathcal Ust_{A/K}}\sum_{w|v}c(A,w|v)\deg(w)=h_{\mathrm{dif}}(A/K).
\end{equation}

Pour tout $v\in\mathcal Ust_{A/K}$ et $w|v$ une place de $L$, soit $B(w|v)=R_{L_w/K_v}(A_{L_w})$ la restriction de Weil de $A_{L_w}$ \`a $K_v$. Soit $\mathfrak D(w|v)$ la diff\'erente de $L_w/K_v$ et $\delta(w|v)=\mathrm{ord}_w(\mathfrak D(w|v))$. Il suit du fait que $L/K$ est une extension galoisienne que $L_w/K_v$ l'est aussi et on note par $G(w|v)$ son groupe de Galois.

La preuve du lemme suivant est inspir\'ee d'un r\'esultat similaire sur les corps de nombres communiqu\'ee par Huajun Lu.

\begin{lemma}\label{lemLu}
Avec les notations pr\'ec\'edentes on a 
$$c(A,w|v)\le c(B(w|v),w|v)\le\delta(w|v)\cdot d.$$
\end{lemma}

\begin{proof}
Soit $\bar K_v$ une cl\^oture alg\'ebrique de $K_v$. Fixons une copie isomor\-phe $F$ de $L_w$ \`a travers un $K_v$-plongement de $L_w$ dans $\bar K_v$. Observons que toute autre image de $L_w$ dans $\bar K_v$ par un $K_v$-plongement est isomorphe \`a $F$. On consid\`ere de cette fa\c con des $K_v$-plongements diff\'erents de $L_w$ dans $F$.

Soit $\Lambda$ l'ensemble des $K_v$-plongements de $L_w$ dans $F$.  Pour tout $\lambda\in\Lambda$, soient $A_{F,\lambda}=A_{K_v}\times_{(\lambda,F)}F$ et 
$$\mathbf B=\prod_{\lambda\in\Lambda}A_{F,\lambda}.$$ 
Cette d\'ecomposition implique une d\'ecomposition d'alg\`ebres de Lie : 
$$\mathrm{Lie}(\mathbf B)=\bigoplus_{\lambda\in\Lambda}\mathrm{Lie}(A_{F,\lambda}).$$
Par la fonctorialit\'e des alg\`ebres de Lie et des restrictions de Weil nous avons des isomorphismes : 
$$\mathrm{Lie}(\mathbf B)\cong\mathrm{Lie}(B(w|v))\otimes_{K_v}F\cong\mathrm{Lie}(A_{L_w})\otimes_{L_w}F.$$ 
Donc, il existe un isomorphisme : 
$$\Psi:\mathrm{Lie}(A_{L_w})\otimes_{L_w}F\longrightarrow\bigoplus_{\lambda\in\Lambda}\mathrm{Lie}(A_{F,\lambda})$$
d\'efini par $t\otimes x\mapsto(x\cdot\lambda_*(t))_\lambda$, o\`u $t\in\mathrm{Lie}(A_{L_w})$, $x\in F$ et $\lambda_*:\mathrm{Lie}(A_{L_w})\to\mathrm{Lie}(A_{F,\lambda})$ induit par le plongement $\lambda$.

Soit $\mathcal O_F$ l'anneau des entiers de $F$. Notons $\mathcal B(w|v)$ le mod\`ele de N\'eron de $B(w|v)$ sur $\mathcal O_v$ et $\mathscr B$ le mod\`ele de N\'eron de $\mathbf B$ sur $\mathcal O_F$. Pour tout $\lambda\in\Lambda$, notons $\mathcal A_{w,F,\lambda}=\mathcal A_w\times_{\mathcal O_w,\lambda}\mathcal O_F$. La d\'ecomposition ant\'erieure de $\mathbf B$ nous donne une autre d\'ecomposition 
$$\mathscr B=\prod_{\lambda\in\Lambda}\mathcal A_{w,F,\lambda}.$$ 
D'o\`u une d\'ecomposition d'alg\`ebres de Lie : 
$$\mathrm{Lie}(\mathscr B)=\bigoplus_{\lambda\in\Lambda}\mathrm{Lie}(\mathcal A_{w,F,\lambda}).$$

Par la propri\'et\'e universelle du mod\`ele de N\'eron, nous avons un morphisme : 
$$\Phi:\mathcal B(w|v)\times_{\mathcal O_v}\mathcal O_F\to\mathscr B,$$ 
et en cons\'equence une application d'alg\`ebres de Lie : 
$$\mathrm{Lie}(\Phi):\mathrm{Lie}(\mathcal B(w|v))\otimes_{\mathcal O_v}\mathcal O_F\to\mathrm{Lie}(\mathscr B).$$ 
On identifie $\mathrm{Lie}(\Phi)$ \`a la version enti\`ere suivante de $\Psi$ : 
$$\Psi':\mathrm{Lie}(\mathcal A_w)\otimes_{\mathcal O_w}\mathcal O_F\to\bigoplus_{\lambda\in\Lambda}\mathrm{Lie}(\mathcal A_{w,F,\lambda}).$$ 
Observons qu'on identifie $\Lambda$ au groupe de Galois $G(w|v)$. D'apr\`es \cite[Proposition 3.3]{liulu} le conoyau de $\Psi'$ est annul\'e par la diff\'erente $\mathfrak D(w|v)$. Alors,
\begin{multline*}
c(B(w|v),w|v)=e(w|v)^{-1}\cdot\mathrm{longueur}(\mathrm{coker}(\Psi'))\le\\ 
e(w|v)^{-1}\cdot\#G(w|v)\cdot d\cdot \mathrm{longueur}(\mathfrak D(w|v))\le\delta(w|v)\cdot d.
\end{multline*}

Pour tout $\lambda\in\Lambda$, soit $p_\lambda:\mathbf B\to A_{F,\lambda}$ la projection sur la $\lambda$-i\`eme composante. La somme $P=\sum_{\lambda\in\Lambda}p_\lambda$ est $G(w|v)$-invariante. Donc, elle nous donne un morphisme $P:B(w|v)\to A_{K_v}$ d\'efini sur $K_v$. Ce morphisme nous donne des morphismes pour les alg\`ebres de Lie des mod\`eles de N\'eron comme suit : 
$$\xymatrix{\mathrm{Lie}(B(w|v))\otimes_{\mathcal O_v}\mathcal O_F\ar[r]\ar[d]^{\Psi'_A}&\mathrm{Lie}(\mathcal A_v)\otimes_{\mathcal O_v}\mathcal O_F\ar[d]^{\Psi'_B}\\ \mathrm{Lie}(\mathscr B)\ar[r]^{P^*}&\mathrm{Lie}(\mathcal A_w).}$$
L'application $P^*$ est d\'efinie par 
$$P^*((x_\lambda)_\lambda)=\sum_{\lambda\in\Lambda}x_\lambda,$$ 
pour tout $(x_\lambda)_\lambda\in\bigoplus_{\lambda\in\Lambda}\mathrm{Lie}(\mathcal A_{w,F,\lambda})$. On en conclut que $P^*$ est surjective. En cons\'equence, 
$$\mathrm{coker}(\Psi'_B)\to\mathrm{coker}(\Psi'_A)$$ est aussi surjective. Donc, 
$$\mathrm{longueur}(\mathrm{coker}(\Psi'_A))\le\mathrm{longueur}(\mathrm{coker}(\Psi'_B)),$$ 
d'o\`u $c(A,w|v)\le c(B(w|v),w|v)$.
\end{proof}

Par un th\'eor\`eme classique dû \`a Dedekind (voir \cite[Chapter III, Theorem 2.6]{Ne}), nous avons l'\'egalit\'e suivante : 
\begin{equation}\label{Ded}
\delta(w|v)=e(w|v)-1+\mathrm{Sw}(w|v),
\end{equation}
o\`u $\mathrm{Sw}(w|v)\ge0$ est un entier qui mesure de la partie sauvage de $\mathfrak D(w|v)$, que l'on note de cette façon pour se souvenir du conducteur de Swan d'une repr\'esentation g\'eom\'etrique. En effet, de fa\c con plus g\'eom\'etrique, si $L$ correspond \`a un rev\^etement fini galoisien $\mathcal X$ de $\mathcal C$ (aussi d\'efini sur $k$), alors 
$$\mathrm{Sw}(w|v)=\mathrm{longueur}(\Omega^1_{\mathcal X/\mathcal C,w})-(e(w|v)-1).$$

Dans notre cas particulier, si $p> 2d+1$, l'extension $L/K$ est mod\'er\'ement ramifi\'ee, donc $\mathrm{Sw}(w|v)=0$. En particulier, par (\ref{ht1}), (\ref{Ded}) et lemme \ref{lemLu}, nous obtenons : 
\begin{equation}\label{ht2}
h_{\mathrm{dif}}(A/K)<h_{\mathrm{st}}(A)+d\cdot\sum_{v\in\mathcal Ust_{A/K}}\sum_{w|v}e(w|v)f(w|v)\deg(v),
\end{equation}
o\`u $f(w|v)$ note le degr\'e d'inertie de $w$ sur $v$.

D'autre part, par la formule de Riemann-Hurwitz et le th\'eor\`eme \ref{corabc}, nous obtenons que : 
\begin{equation}\label{bdst}
h_{\mathrm{st}}(A)\le p^e\left(\frac{d-d_0}2\cdot(2g-2+\bar s)+p^{2d}\right).
\end{equation}
Donc, par (\ref{ht2}) et (\ref{bdst}), nous concluons que : 
\begin{equation}\label{abc gen}
h_{\mathrm{dif}}(A/K)<p^e\left(\frac{d-d_0}2\cdot(2g-2+\bar s)+p^{2d}\right)+d\cdot[L:K]\cdot\bar s.
\end{equation}

Nous pouvons choisir $\ell=2$ (car $p>2d+1$, en particulier $p\neq 2$) et borner $[L:K]\le 2^{4d^2}$. Nous avons donc prouv\'e le théorème \ref{thmabc}.

\section{R\'esultats locaux}
\subsection{Bonne r\'eduction potentielle}
Supposons que $J$ ait bonne r\'eduction potentielle en une place $v$ de $K$, \textit{i.e.} il existe une extension finie $L_w$ de $K_v$ telle que $J_{L_w}=J\times_{K_v}L_w$ ait bonne r\'eduction. Apr\`es une extension de $L_w$, on peut supposer que $L_w/K_v$ soit galoisienne de groupe $G(w|v)$ d'ordre premier \`a $p$.

\subsubsection{Rappel sur les sch\'emas en groupes plats d'ordre $p$}
Soit $Y$ un sch\'ema en caract\'eristique $p$. La donn\'ee d'un sch\'ema en groupes plat $N_{a,b}^{\mathcal L}$ d'ordre $p$ sur $Y$ \'equivaut \`a la donn\'ee d'un triplet $(\mathcal L,a,b)$, o\`u $\mathcal L$ est un faisceau inversible sur $Y$, $a\in H^0(Y,\mathcal L^{\otimes p-1})$ et $b\in H^0(Y,\mathcal L^{\otimes 1-p})$ tels que $a\otimes b=0$. Le cas o\`u $a=b=0$ correspond \`a $\alpha_p$, le cas o\`u $a=1$ et $b=0$ \`a $\mathbb Z/p\mathbb Z$ et le cas $a=0$ et $b=1$ \`a $\mu_p$ (\emph{confer} \cite[chapter III, 0.9]{mil}). 

\subsubsection{Filtrations}
Pour tout entier $i\ge0$, d\'efinissons 
$$U^{[i]}_{K_v}=\{\bar f\in K_v^*/K_v^{*p}\,|\,\mathrm{ord}_v(1-f)\ge i\}.$$
Les $U_{K_v}^{[i]}$ forment une filtration d\'ecroissante exhaustive du groupe compact $K_v^*/K_v^{*p}$ par sous-groupes d'indice fini avec $(K_v^*/K_v^{*p})/U_{K_v}^{[0]}\cong\mathbb Z/p\mathbb Z$ (de façon canonique), et $U_{K_v}^{[i]}/U_{K_v}^{[i+1]}\cong k$ (de façon non canonique, si $p\nmid i$), et $U_{K_v}^{[pi]}/U_{K_v}^{[pi+1]}\cong \{1\}$. 

\subsubsection{Groupes de Selmer}
Soit $\mathcal O_{L_w}$ l'anneau de valuation local de $L_w$. Notons par $\mathrm{ord}_{w}(\cdot)$ la valuation associ\'ee \`a $\mathcal O_{L_w}$. Soit $n_w=-\mathrm{ord}_{w}(\mathcal D(\omega_{\mathcal J/\mathcal C}))$, o\`u $\mathcal J\to\mathcal C$ est le mod\`ele de N\'eron de $J/K$ sur $\mathcal C$. Par les propri\'et\'es (\ref{frobenius}) et (\ref{prop2}) du Frobenius et des groupes de Selmer et par les calculs locaux de \cite[III.7.5]{mil}, il vient 
\begin{multline*}
\mathrm{Sel}_{J_{L_w}}(L_w,F)\cong H^1(\mathcal O_{L_w},\ker(F))\cong H^1(\mathcal O_{L_w},\mu_p)^{\oplus a}\oplus H^1(\mathcal O_{L_w},G)\\
\cong (U_{L_w}^{[pn_w]})^{\oplus a}\oplus(\mathcal O_{L_w}/\mathcal O_{L_w}^p)^{\oplus(2d-a)}_{\mathrm{sd}},
\end{multline*}
o\`u le dernier produit est un produit semi-direct de $2d-a$ copies de $\mathcal O_{L_w}/\mathcal O_{L_w}^p$. Maintenant, il suffit de prendre les invariants par le groupe $G(w|v)$ et obtenir
\begin{equation}\label{Selmer1}
\mathrm{Sel}_{J}(K_v,F)\cong\mathrm{Sel}_{J_{L_w}}(L_w,F)^{G(w|v)}
\cong (U_{K_v}^{[i_v]})^{\oplus a}\oplus\mathcal (O_{K_v}/\mathcal O_{K_v}^p)^{\oplus(2d-a)}_{\mathrm{sd}},
\end{equation}
o\`u le dernier produit est un produit semi-direct de $2d-a$ copies de $\mathcal O_{K_v}/\mathcal O_{K_v}^p$ et $i_v=-p\cdot\mathrm{ord}_v(\mathcal D(\omega_{\mathcal J/\mathcal C}))$.

\subsection{R\'eduction semi-ab\'elienne potentielle}
Rappelons d'abord ce qu'est l'uniformisation de Raynaud. Soit $L_w$ une extension finie et galoisienne de $K_v$ de groupe $G$ d'ordre premier \`a $p$ o\`u $J$ a r\'eduction semi-ab\'elienne d\'eploy\'ee. Cela veut dire qu'il existe une vari\'et\'e semi-ab\'elienne $G$ d\'efinie sur $L_w$ et un r\'eseau $\Lambda\subset G(L_w)$ tel que 
$$J(L_w)\cong G(L_w)/\Lambda$$ 
dans le cadre de la g\'eom\'etrie rigide (voir \cite[chapter 6]{frvp}). Rappelons que $G$ est d\'efinie par la suite exacte : 
$$0\to\mathbb G_m^t(L_w)\to G\to B\to0,$$
o\`u $d=\dim J=\dim G=t+\dim B$.

La construction de la partie \ref{frob} de changement de base par l'automorphisme de Frobenius nous donne encore une autre vari\'et\'e semi-ab\'elienne $G^{(p)}$ d\'efinie sur $L_w$ et un r\'eseau $\Lambda^{(p)}$ (dont les g\'en\'erateurs sont obtenus en prenant la puissance $p$-i\`eme des g\'en\'erateurs de $\Lambda$). Dans cette situation on a un isomorphisme rigide 
$$J^{(p)}(L_w)\cong G^{(p)}(L_w)/\Lambda^{(p)}.$$

Consid\'erons les applications 
$$\mathbb G_m^t(L_w)/\Lambda^{(p)}\to G^{(p)}(L_w)/\Lambda^{(p)}\cong J^{(p)}(L_w)\to H^1(L_w,\ker F),$$
dont la compos\'ee est surjective et la premi\`ere application est injective. Par convention la classe $\bar 0$ se rel\`eve en la classe $\bar 1$. En particulier, nous obtenons que 
\begin{multline*}
\mathrm{Sel}_{J_{L_w}}(L_w,F)\cong H^1(L_w,\ker(F))\cong 
H^1(L_w,\mu_p)^{\oplus a}\oplus H^1(L_w,G)\\
\cong (L_w^*/L_w^{*p})^{\oplus a}\oplus (L_w/L_w^p)^{\oplus(2d-a)}_{\mathrm{sd}},
\end{multline*}
o\`u le dernier produit est un produit semi-direct de $2d-a$ copies de $L_w/L_w^p$. \emph{A fortiori}, en prenant les invariants par $G(w|v)$, 
\begin{equation}\label{Selmer2}
\mathrm{Sel}_{J}(K_v,F)\cong(K_v^*/K_v^{*p})^{\oplus a}\oplus (K_v/K_v^p)^{\oplus(2d-a)}_{\mathrm{sd}}.
\end{equation}

\section{R\'esultats globaux}
\subsection{Calcul du groupe de Selmer}
On note $\mathbf M$ l'ensemble des places de $K$ o\`u $J$ a mauvaise r\'eduction potentielle et $\mathbf B$ l'ensemble des places de $K$ o\`u $J$ a bonne r\'eduction potentielle. Soit 
$$D=\sum_{v\in\mathbf M}[v]-\sum_{v\in\mathbf B}i_v\cdot[v]\in\mathrm{Div}\,\mathcal C,$$
o\`u $i_v=-p\cdot\mathrm{ord}_v(\mathcal D(\omega_{\mathcal J/\mathcal C}))$. Soit $\bar s$ le nombre de points g\'eom\'etriques de $\mathcal C$ o\`u $J$ a mauvaise r\'eduction. Observons que par la d\'efinition du diviseur diff\'erentiel, on a
$$0<\deg D\le\bar s+p\cdot h_{\mathrm{dif}}(J/K).$$

Nous disposons de plus d'un op\'erateur $1/p$-lin\'eaire, appel\'e op\'erateur de Cartier et not\'e $\mathscr C$, qui agit sur les diff\'erentielles $\Omega^1_K$ de $K$ (voir \cite{serremex}). Les diff\'erentielles fix\'ees par $\mathscr C$ sont les diff\'erentielles logarithmiques $df/f$ pour $f\in K^*$. 

Les deux r\'esultats locaux (\ref{Selmer1}) et (\ref{Selmer2}) de la section 4 nous permettent de formuler le th\'eor\`eme suivant, dont la d\'emonstration se r\'eduit aux cas locaux d\'ej\`a trait\'es. 

\begin{theorem} Nous avons un isomorphisme
\begin{equation}\label{seliso}
\mathrm{Sel}_J(K,F)\cong (H^0(\mathcal C,\Omega^1_{\mathcal C}(D))^{\mathscr C})^{\oplus a} \oplus(H^0(\mathcal C,\Omega^1_{\mathcal C}(D))^{\mathscr C=0})^{\oplus(2d-a)}.
\end{equation}
\end{theorem}

\begin{proof}
Nous avons une application injective $K^*/K^{*p}\hookrightarrow\Omega^1_K$ d\'efinie par $\bar f\mapsto df/f$. L'image de cette application est exactement l'ensemble des diff\'erentielles qui sont fix\'ees par l'op\'erateur de Cartier $\mathscr C$. 

D'autre part, nous avons aussi une applications injective $K/K^p\hookrightarrow\Omega^1_K$ d\'efinie par $\bar f\mapsto df$ et dont l'image c'est le noyau de $\mathscr C$. 

Par (\ref{Selmer2}), pour toute place $v$ de $K$, un \'el\'ement $\xi\in\mathrm{Sel}(K_v,F)$ est repr\'esent\'e par un $2d$-uplet $(\xi_{1,v},\cdots,\xi_{2d,v})$, o\`u les premi\`eres $a$ coordonn\'ees sont des classes dans $K_v^*/K_v^{*p}$, et les derni\`eres $2d-a$ coordonn\'ees sont des classes dans $K_v/K_v^p$. 

Pour tout $v\in\mathbf B$ et $1\le i\le a$, pour que $\xi\in\mathrm{Sel}(K_v,F)$ il faut et il suffit que $\mathrm{ord}_v(df_i/f_i)\ge i_v$, o\`u $\xi_i=\bar f_i$. Pour $a+1\le i\le 2d$, la classe $\xi=\bar f_i$ appartient au groupe de Selmer local si et seulement si $\mathrm{ord}_v(df_i)\ge0$. Ceci suffit pour obtenir le r\'esultat.
\end{proof}

\begin{corollary}\label{bigselmer}
Nous avons une application injective 
$$\mathrm{Sel}(K,F)\hookrightarrow H^0(\mathcal C,\Omega^1_{\mathcal C}(D))^{\oplus2d},$$
et, \textit{a fortiori},

\begin{equation}\label{Selmer borne}
\#\mathrm{Sel}(K,F)\le2d\cdot q^{g-1+\bar s+p\cdot h_{\mathrm{dif}}(J/K)}.
\end{equation}
\end{corollary}

\subsection{Application du th\'eor\`eme abc aux groupes de Selmer}
\begin{corollary}\label{bornesel}
$$\#\mathrm{Sel}(K,F)<c_3=2d\cdot q^{g-1+p\cdot c_2},\quad\text{o\`u}$$
$$c_2=p^e\cdot\left(\frac{d}2\cdot(2g-2+f_{\mathcal X/\mathcal C})+p^{2d}\right)+d\cdot2^{4d^2}\cdot f_{\mathcal X/\mathcal C}.$$
\end{corollary}

\begin{proof}
En posant, 
\begin{equation}\label{c0}
c_0=p^e\cdot\left(\frac{d}2\cdot(2g-2+\bar s)+p^{2d}\right)+d\cdot2^{4d^2}\cdot\bar s,
\end{equation}
en appliquant $(\ref{Selmer borne})$ et en bornant $h_{\mathrm{dif}}(J/K)$ comme dans le th\'eor\`eme \ref{thmabc}, nous arrivons \`a l'in\'egali\-t\'e : 
$$\#\mathrm{Sel}(K,F)<c_1=2d\cdot q^{g-1+\bar s+p\cdot c_0}.$$
Finalement, comme 
$$\bar s\le f_{J/K}\le f_{\mathcal X/\mathcal C},$$
(voir \cite[Proposition 2.8]{PaPa}), nous en concluons la preuve du corollaire.
\end{proof}

\section{$F$-descente}

On d\'emontre dans cette partie l'\'enonc\'e principal.  A partir d'ici, on va avoir besoin d'un corps de rationalit\'e des points de $X\cap \Gamma$. On d\'efinit donc :

\begin{definition}
Soit $E_\Gamma$ la plus petite extension alg\'ebrique de $K$ telle que les points de $X\cap\Gamma$ soient d\'efinis sur $E_\Gamma$. On notera $J_{\Gamma}=J\times_K E_{\Gamma}$.
\end{definition}

\begin{proof}[D\'emonstration du th\'eor\`eme \ref{papa principal}]
Supposons d'abord que $X$ soit d\'efinie sur $K^p$, mais qu'elle ne soit pas d\'efinie sur $K^{p^2}$. Il existe donc une courbe lisse, compl\`ete et g\'eom\'etrique\-ment connexe $X_1$ qui est d\'efinie sur $K$, mais qui n'est pas d\'efinie sur $K^p$ et telle que $X_1^{(p)}\cong X$. Notons par $F:X_1\to X$ le Frobenius relatif de $X_1$. De façon similaire, si $J_1$ est la vari\'et\'e jacobienne de $X_1$, nous avons $J_1^{(p)}\cong J$ et le morphisme de Frobenius relatif de $J_1$ est $F:J_1\to J$. 

Soit $\Gamma_1$ un sous-groupe de $F^{-1}(\Gamma)$. Observons qu'il est un sous-groupe de $J_1(K_s)$ tel que $\Gamma_1/p\Gamma_1$ soit fini de cardinal au plus $\#\Gamma/p\Gamma$.  Par \cite{BuVo} on a la borne
$$\#(X_1\cap\Gamma_1)\le\#(\Gamma/p\Gamma)\cdot(3p)^d\cdot(8d-2)\cdot d!.$$
Par construction, l'extension alg\'ebrique minimale $E_\Gamma$ de $K$ telle que les points de  $X\cap\Gamma$ sont d\'efinis sur $E_\Gamma$, est aussi l'extension alg\'ebrique minimale de $K$ telle que les points de $X_1\cap\Gamma_1$ sont d\'efinis sur $E_\Gamma$. De plus, comme $X_1\cap\Gamma_1$ est fini, l'extension $E_\Gamma/K$ est aussi finie. Notons par $g_\Gamma$ son genre. Si $H$ est un sous-groupe d'indice fini d'un groupe $G$, notons par $(G:H)$ l'indice de $H$ dans $G$. Observons que nous avons les in\'egalit\'es : 

\begin{multline*}
(\#X\cap \Gamma) / (\#F(X_1\cap\Gamma_1)) \leq (\#X(E_\Gamma)) / \#(F(X_1(E_\Gamma)))\le\\ 
(J(E_\Gamma):F(J_1(E_\Gamma)))\le\#\mathrm{Sel}_{J_\Gamma}(E_\Gamma,F).
\end{multline*}

Soit $\epsilon\ge0$ le plus grand entier tel que $X$ soit d\'efinie sur $E_\Gamma^{p^{\epsilon}}$, mais $X$ ne soit pas d\'efinie sur $E_\Gamma^{p^{\epsilon+1}}$. Nous avons l'in\'egalit\'e $\epsilon\le e$. 

Soit $\mathcal C_\Gamma$ la courbe lisse g\'eom\'etriquement connexe et compl\`ete d\'efinie sur $k$ de genre $g_\Gamma$ telle que $k(\mathcal C_\Gamma)=E_\Gamma$. Soit $\bar s_\Gamma$ le nombre de points g\'eom\'etriques de $\mathcal C_\Gamma$ o\`u $J_{\Gamma}=J\times_K E_\Gamma$ a mauvaise r\'eduction. Observons que nous avons les in\'egalit\'es
\begin{equation}\label{points}
\bar s_\Gamma\le f_{J_\Gamma/E_\Gamma}\le[E_\Gamma:K]\cdot f_{J/K}\le[E_\Gamma:K]\cdot f_{\mathcal X/\mathcal C}
\end{equation}
La deuxi\`eme in\'egalit\'e suit de \cite[Proposition 3.7]{Pa}. La troisi\`eme in\'egalit\'e suit de \cite[Proposition 2.8]{PaPa}. Par la d\'emonstration du corollaire \ref{bornesel}, on doit remplacer $c_0$ par : 
$$c_4=p^e\cdot\left(\frac{d}2\cdot(2g_\Gamma-2+\bar s_\Gamma)+p^{2d}\right)+d\cdot2^{4d^2}\cdot\bar s_\Gamma.$$
En employant (\ref{points}), on obtient la majoration suivante : 
$$c_4\le c_5=[E_\Gamma:K]\cdot\left(p^e\cdot\left(\frac{d}2\cdot(2g_\Gamma+f_{\mathcal X/\mathcal C})+p^{2d}\right)+d\cdot2^{4d^2}\cdot f_{\mathcal X/\mathcal C}\right).$$
On conclut par le corollaire \ref{bornesel} que : 
\begin{equation}\label{borneselga}
\#\mathrm{Sel}(E_\Gamma,F)<c_6=2d\cdot q^{g_\Gamma-1+p\cdot c_5}.
\end{equation}
Donc, 
$$\#(X\cap\Gamma)\le C_{\mathrm{BV}}\cdot c_6.$$

Maintenant si $X$ est d\'efinie sur $K^{p^2}$, mais n'est pas d\'efinie sur $K^{p^3}$, comme auparavant nous avons deux courbes $X_1$ et $X_2$ telles que 
$$X_2\overset F\longrightarrow X_1=X_2^{(p)}\overset F\longrightarrow X=X_1^{(p)}=X_2^{(p^2)}.$$
Notons par $J_2$ la vari\'et\'e jacobienne de $X_2$. Dans ce cas nous avons les in\'egalit\'es suivantes : 
$$\begin{aligned}
\#(X_1\cap\Gamma_1)&\le\#(X_2\cap\Gamma_2)\cdot(J_1(E_\Gamma):F(J_2(E_\Gamma)))\\ 
&\le\#(X_2\cap\Gamma_2)\cdot\#\mathrm{Sel}_{(J_1)_{E_\Gamma}}(E_\Gamma,F)\\
\#(X\cap\Gamma)&\le\#(X_1\cap\Gamma_1)\cdot(J(E_\Gamma):F(J_1(E_\Gamma)))\\
&\le\#(X_2\cap\Gamma_2)\cdot(\#\mathrm{Sel}_{J_{E_\Gamma}}(E_\Gamma,F))^2.
\end{aligned}$$
Dans la premi\`ere in\'egalit\'e, nous employons le fait que le morphisme $F$ est purement ins\'eparable, donc $E_\Gamma$ reste la plus petite extension finie de $K$ telle que $X_2\cap\Gamma_2$ est d\'efini sur $E_\Gamma$. Comme auparavant, on prend un  sous-groupe $\Gamma_2$ de $F^{-1}(\Gamma_1)$, o\`u $F$ d\'esigne l'isog\'enie $F:J_2\to J_1$. Le seul terme dans les majorants qui d\'epend des vari\'et\'es jacobiennes, c'est leur conducteur. Mais ces vari\'et\'es, en tant que vari\'et\'es ab\'eliennes, sont isog\`enes, donc leurs conducteurs coïncident. D'o\`u
$$\#(X\cap\Gamma)\le C_{\mathrm{BV}}\cdot c_6^2.$$
Maintenant, une induction assez facile nous montre le th\'eor\`eme (pour un argument similaire, voire la d\'emonstration de \cite[lemma 3.3]{PaPa}).
\end{proof}

\begin{remark}\label{remtame}
Si l'extension $E_\Gamma/K$ est mod\'erement ramifi\'ee, on peut remplacer $c_5$ par : 
$$c_{5,t}=[E_\Gamma:K]\cdot\left(p^e\cdot\left(\frac{d}2\cdot(2g-1+f_{\mathcal X/\mathcal C})+p^{2d}\right)+d\cdot2^{4d^2}\right),$$
et $c_6$ par : 
$$c_{6,t}=2d\cdot q^{c_{7,t}+p\cdot c_{5,t}},\quad\text{o\`u}\quad c_{7,t}=[E_\Gamma:K]\cdot(2g-1).$$
Dans ce cas la borne sup\'erieure pour le cardinal de $X\cap\Gamma$ ne d\'epend plus de $g_\Gamma$, mais juste de $[E_\Gamma:K]$.
\end{remark}

\begin{remark}
Pour retrouver en corollaire le r\'esultat de \cite{PaPa} il suffit de sp\'ecialiser le groupe $\Gamma=J(K)$. On trouve alors bien s\^ur $[E_\Gamma:K]=1$.
\end{remark}

\section{Comparaison avec le cas des corps de nombres}

Lorsqu'on s'int\'eresse au cas des corps de nombres, on dispose d'un \'enonc\'e g\'en\'eral d\^u \`a R\'emond.

\begin{theorem}[R\'emond]\label{rem} \cite[Th\'eor\`eme 1.2]{Rem}
Soient $A$ une vari\'et\'e ab\'e\-lienne sur $\overline{\mathbb{Q}}$ de dimension $d$ et $\mathcal{L}$ un faisceau sym\'etrique et ample sur $A$. Soient $X$ un sous-sch\'ema ferm\'e de $A$ de dimension $m$ et $\Gamma$ un sous-groupe de rang $r\in{\mathbb{N}}$ de $A(\overline{\mathbb{Q}})$. Il existe une constante $c(A,\mathcal{L})>0$ (d\'ependant de la hauteur de $A$), un entier naturel $S$, des \'el\'ements $x_1, ..., x_S$ de $X(\overline{\mathbb{Q}})\cap \Gamma$ et des sous-vari\'et\'es ab\'eliennes $B_1, ..., B_S$ de $A$ de sorte que $x_i+B_i\subset X$ si $1\leq i \leq S$ et $$X(\overline{\mathbb{Q}})\cap \Gamma=\bigcup_{i=1}^S(x_i+B_i)(\overline{\mathbb{Q}})\cap\Gamma,$$
avec $$S\leq (c(A,\mathcal{L})\cdot\deg_{\mathcal{L}}X)^{(r+1)d^{5(m+1)^2}}.$$
\end{theorem}

Si on sp\'ecialise dans le cas des courbes $X$ de genre $d\geq 2$ on a $m=1$ et les $B_i$ sont toutes triviales car $X$ n'est pas une courbe elliptique. Notons par $J$ la vari\'et\'e jacobienne de $X$.  On obtient ainsi 
$$S\leq (c_1(J,\mathcal{L}))^{(r+1)d^{20}},$$ 
o\`u $c_1(J,\mathcal{L})$ d\'epend de mani\`ere polynomiale de la hauteur de la vari\'et\'e $J$, comme explicit\'e dans l'article \cite{DavPhi} page 641 par les th\'eor\`emes 1.3 et 1.4.

Comparons ce r\'esultat avec la situation o\`u $K$ est un corps de fonctions en une variable sur un corps fini et $\Gamma$ un groupe de rang fini $r$. Alors le groupe quotient $\Gamma/p\Gamma$ est en particulier de cardinal fini et on gagne de plus un contr\^ole plus explicite sur son cardinal $\#(\Gamma/p\Gamma)$. En effet ce nombre est au plus \'egal \`a $p^{r+d}$, o\`u $d\ge2$ est le genre de $X$. Dans le cas des corps de fonctions, la borne ne d\'epend donc pas de la hauteur de la vari\'et\'e jacobienne, mais uniquement du conducteur de la courbe et des invariants classiques que sont les degr\'es de corps, genres de courbes et degr\'e d'ins\'eparabilit\'e.

La d\'ependance en la hauteur diff\'erentielle de $J$ surgit pourtant dans la preuve au niveau de la borne lorsqu'on veut consid\'erer des courbes $X$ qui peuvent \^etre d\'efinies sur $K^{p^e}$ pour un entier $e\ge1$. Dans ce cas l\`a le r\'esultat \cite[Theorem 5.3]{HiPa} nous permet de la remplacer par le conducteur de $J/K$ et d'apr\`es \cite[Proposition 2.8]{PaPa} ce nombre est au plus le conducteur d'un mod\`ele $\mathcal X/\mathcal C$ de $X/K$ sur $\mathcal C$.

Un point important cependant : la borne de R\'emond ne fait pas intervenir le corps de rationalit\'e des points d'intersection, elle est donc plus fine en ce sens.

Si on se concentre sur le cas $\Gamma=J(K)$, o\`u $K$ est un corps de nombres, on s'attend en fait \`a un r\'esultat plus fort. La conjecture suivante, dont on peut trouver une justification dans l'article \cite{TdD} va dans ce sens.

\begin{conjecture}
Soit $d\geq 2$ et $D\geq 1$ deux entiers naturels. Alors il existe une constante $c(d,D)>0$ telle que pour tout corps de nombres $K$ de degr\'e $D$, pour toute courbe $X/K$ de genre $d\geq 2$, on a la majoration suivante
$$\#X(K)\leq c(d,D)^{1+\mathrm{rang}(J(K))}.$$
\end{conjecture}

Il n'y aurait alors plus de d\'e\-pen\-dan\-ce en la hauteur de la jacobienne. On trouvera dans \cite{Paz13} une telle borne d\'emontr\'ee pour une famille de courbes de genre $2$ ayant potentiellement bonne r\'eduction partout. La d\'epen\-dance en le rang de la jacobienne peut au besoin \^etre troqu\'ee contre une d\'ependance en le conducteur de la jacobienne (ou m\^eme de la courbe sous-jacente) en utilisant par exemple le th\'eor\`eme 5.1 de \cite{rem10} page 775 qui majore le rang d'une vari\'et\'e ab\'elienne par son conducteur. 

Notons qu'il existe une conjecture plus audacieuse sur l'existence d'une borne sur le nombre de points rationnels $\#X(K)$ dans laquelle le rang de la jacobienne de $X$ n'est plus pr\'esent dans le majorant, voir \`a ce sujet l'article \cite{CaHaMa}.

\end{document}